\documentclass[12pt]{amsart}
\usepackage{a4,amsthm, epsf, amscd}
\usepackage{amssymb} 
\usepackage{a4, amsthm} 

\theoremstyle{plain}
\newtheorem{theorem}{Theorem}[section]
\newtheorem{lemma}[theorem]{Lemma}
\newtheorem{prop}[theorem]{Proposition}

\theoremstyle{definition}

\newtheorem{remark}[theorem]{Remark} 

\font\ninerm=cmr9  \font\eightrm=cmr8  \font\sixrm=cmr6
\font\ninei=cmmi9  \font\eighti=cmmi8  \font\sixi=cmmi6
\font\ninesy=cmsy9 \font\eightsy=cmsy8 \font\sixsy=cmsy6
\font\ninebf=cmbx9 \font\eightbf=cmbx8 \font\sixbf=cmbx6
\font\nineit=cmti9 \font\eightit=cmti8 
\font\ninett=cmtt9 \font\eighttt=cmtt8 
\font\ninesl=cmsl9 \font\eightsl=cmsl8

\font\twelverm=cmr12 at 15pt
\font\twelvei=cmmi12 at 15pt
\font\twelvesy=cmsy10 at 15pt
\font\twelvebf=cmbx12 at 15pt
\font\twelveit=cmti12 at 15pt
\font\twelvett=cmtt12 at 15pt
\font\twelvesl=cmsl12 at 15pt
\font\twelvegoth=eufm10 at 15pt

\font\tengoth=eufm10  \font\ninegoth=eufm9
\font\eightgoth=eufm8 \font\sevengoth=eufm7 
\font\sixgoth=eufm6   \font\fivegoth=eufm5
\newfam\gothfam \def\goth{\fam\gothfam\tengoth} 
\textfont\gothfam=\tengoth
\scriptfont\gothfam=\sevengoth 
\scriptscriptfont\gothfam=\fivegoth

\catcode`@=11
\newskip\ttglue

\def\tenpoint{\def\rm{\fam0\tenrm}
  \textfont0=\tenrm \scriptfont0=\sevenrm
  \scriptscriptfont0\fiverm
  \textfont1=\teni \scriptfont1=\seveni
  \scriptscriptfont1\fivei 
  \textfont2=\tensy \scriptfont2=\sevensy
  \scriptscriptfont2\fivesy 
  \textfont3=\tenex \scriptfont3=\tenex
  \scriptscriptfont3\tenex 
  \textfont\itfam=\tenit\def\it{\fam\itfam\tenit}%
  \textfont\slfam=\tensl\def\sl{\fam\slfam\tensl}%
  \textfont\ttfam=\tentt\def\tt{\fam\ttfam\tentt}%
  \textfont\gothfam=\tengoth\scriptfont\gothfam=\sevengoth 
  \scriptscriptfont\gothfam=\fivegoth
  \def\goth{\fam\gothfam\tengoth}
  \textfont\bffam=\tenbf\scriptfont\bffam=\sevenbf
  \scriptscriptfont\bffam=\fivebf
  \def\bf{\fam\bffam\tenbf}%
  \tt\ttglue=.5em plus.25em minus.15em
  \normalbaselineskip=12pt \setbox\strutbox\hbox{\vrule
  height8.5pt depth3.5pt width0pt}%
  \let\big=\tenbig\normalbaselines\rm}

\def\ninepoint{\def\rm{\fam0\ninerm}
  \textfont0=\ninerm \scriptfont0=\sixrm
  \scriptscriptfont0\fiverm
  \textfont1=\ninei \scriptfont1=\sixi
  \scriptscriptfont1\fivei 
  \textfont2=\ninesy \scriptfont2=\sixsy
  \scriptscriptfont2\fivesy 
  \textfont3=\tenex \scriptfont3=\tenex
  \scriptscriptfont3\tenex 
  \textfont\itfam=\nineit\def\it{\fam\itfam\nineit}%
  \textfont\slfam=\ninesl\def\sl{\fam\slfam\ninesl}%
  \textfont\ttfam=\ninett\def\tt{\fam\ttfam\ninett}%
  \textfont\gothfam=\ninegoth\scriptfont\gothfam=\sixgoth 
  \scriptscriptfont\gothfam=\fivegoth
  \def\goth{\fam\gothfam\tengoth}
  \textfont\bffam=\ninebf\scriptfont\bffam=\sixbf
  \scriptscriptfont\bffam=\fivebf
  \def\bf{\fam\bffam\ninebf}%
  \tt\ttglue=.5em plus.25em minus.15em
  \normalbaselineskip=11pt \setbox\strutbox\hbox{\vrule
  height8pt depth3pt width0pt}%
  \let\big=\ninebig\normalbaselines\rm}

\def\eightpoint{\def\rm{\fam0\eightrm}
  \textfont0=\eightrm \scriptfont0=\sixrm
  \scriptscriptfont0\fiverm
  \textfont1=\eighti \scriptfont1=\sixi
  \scriptscriptfont1\fivei 
  \textfont2=\eightsy \scriptfont2=\sixsy
  \scriptscriptfont2\fivesy 
  \textfont3=\tenex \scriptfont3=\tenex
  \scriptscriptfont3\tenex 
  \textfont\itfam=\eightit\def\it{\fam\itfam\eightit}%
  \textfont\slfam=\eightsl\def\sl{\fam\slfam\eightsl}%
  \textfont\ttfam=\eighttt\def\tt{\fam\ttfam\eighttt}%
  \textfont\gothfam=\eightgoth\scriptfont\gothfam=\sixgoth 
  \scriptscriptfont\gothfam=\fivegoth
  \def\goth{\fam\gothfam\tengoth}
  \textfont\bffam=\eightbf\scriptfont\bffam=\sixbf
  \scriptscriptfont\bffam=\fivebf
  \def\bf{\fam\bffam\eightbf}%
  \tt\ttglue=.5em plus.25em minus.15em
  \normalbaselineskip=9pt \setbox\strutbox\hbox{\vrule
  height7pt depth2pt width0pt}%
  \let\big=\eightbig\normalbaselines\rm}

\def\twelvepoint{\def\rm{\fam0\twelverm}
  \textfont0=\twelverm\scriptfont0=\tenrm
  \scriptscriptfont0\sevenrm
  \textfont1=\twelvei\scriptfont1=\teni
  \scriptscriptfont1\seveni 
  \textfont2=\twelvesy\scriptfont2=\tensy
  \scriptscriptfont2\sevensy 
   \textfont\itfam=\twelveit\def\it{\fam\itfam\twelveit}%
  \textfont\slfam=\twelvesl\def\sl{\fam\slfam\twelvesl}%
  \textfont\ttfam=\twelvett\def\tt{\fam\ttfam\twelvett}%
  \textfont\gothfam=\twelvegoth\scriptfont\gothfam=\ninegoth 
  \scriptscriptfont\gothfam=\sevengoth
  \def\goth{\fam\gothfam\twelvegoth}
  \textfont\bffam=\twelvebf\scriptfont\bffam=\ninebf
  \scriptscriptfont\bffam=\sevenbf
  \def\bf{\fam\bffam\twelvebf}%
  \tt\ttglue=.5em plus.25em minus.15em
  \normalbaselineskip=12pt \setbox\strutbox\hbox{\vrule
  height9pt depth4pt width0pt}%
  \let\big=\twelvebig\normalbaselines\rm}

\def\nsmallskip{\smallskip\noindent}
\def\bbigskip{\bigskip\bigskip}
\def\nbigskip{\bigskip\noindent}

\def\nmedskip{\medskip\noindent}
\def\buildunder#1#2{\mathrel{\mathop{\kern0pt #2}
\limits_{#1}}}
\def\buildover#1#2{\buildrel#1\over#2}

\def\qq{/\kern-.185em /}

\def\REM #1{}

\begin{document}

\title[On hyperbolicity]{On hyperbolicity of SU(2)-equivariant, punctured disc bundles
over the complex affine quadric}

\bbigskip

\author[Iannuzzi]{A. Iannuzzi}

\address{Andrea Iannuzzi: Dip.\ di Matematica,
Universit\`a di Roma  ``Tor Vergata", Via della Ricerca Scientifica,
I-00133 Roma, Italy.} 
\email{iannuzzi@mat.uniroma2.it}



\begin{abstract}
Given a holomorphic line bundle over the complex affine quadric
$\,Q^2$, we investigate its Stein, $\,SU(2)$-equivariant  disc
bundles. Up to equivariant
biholomorphism, these are all contained in a maximal one,
say $\,\Omega_{max}$. By removing the zero section from
 $\,\Omega_{max}\,$
one obtains the unique Stein, $\,SU(2)$-equivariant, punctured disc
bundle over $\,Q^2\,$ which contains entire curves.
All other such  punctured disc bundles 
are shown to be  Kobayashi hyperbolic.
\end{abstract}
 
\maketitle


\section{Introduction} 

\bigskip
Consider a Reinhardt domain $\,D_\rho\,$ in $\,{\mathbb{C}} \times {\mathbb{C}}^*\,$ of 
the form
$$\,\{\,(w,z)\in {\mathbb{C}} \times {\mathbb{C}}^* \ : \ |z| \rho(|w|) < 1 \,\},$$
 where $\,\rho:{\mathbb{R}} \to {\mathbb{R}}^{>0}\,$ is an even, 
 upper semicontinuous function.
Let $\,S^1\,$
act on $\,D_\rho\,$ by $\,e^{it}\cdot (w,z):=  (e^{it} w,z)$.
Then $\,D_\rho\,$ can be regarded
as an $\,S^1$-invariant punctured disc bundle over $\,{\mathbb{C}}$, with
$\,S^1$-equivariant projection
$\,D_\rho \to {\mathbb{C}}\,$ given by $\,(w,z) \to w$.
By rescaling the fiber coordinate one can normalize every
 $\,D_\rho\,$ so  that $\,\rho(0)=1$.

Note that $\,D_\rho\,$ is Stein if and only if $\,\rho\,$ is logarithmically convex,
i.e. if $\,\log  \rho\,$ is convex. 
Under this assumption one has the
extremal case $\,\rho \equiv 1$, corresponding to the
trivial punctured disc bundle $\,D_{max}= {\mathbb{C}} \times \Delta^*$.
Here $\,\Delta^*\,$ denotes the punctured unit disc  in $\,{\mathbb{C}}$.
All other  Stein, normalized, punctured disc bundles 
   are contained in $\,D_{max}$.
These correspond to non constant, logarithmically convex $\,\rho\,$ with
$\,\rho(0)=1$. In particular 
$\,\lim \rho(h) = \infty\,$ as $\,h \to \infty\,$ which, 
by a simple argument,  implies that every
non-maximal, Stein, punctured disc bundle $\,D_\rho\,$
is  Kobayashi hyperbolic.
Then, by a result of Zwonek (\cite{Zwo}),  one also 
knows that $\,D_\rho\,$ is biholomorphic to a bounded Reinhardt domain.

 Let $\,U^{\mathbb{C}}=SL(2,{\mathbb{C}})\,$ 	and 
$\,K^{\mathbb{C}} \,$ be the universal complexifications of
$$\,U :=SU(2) \quad {\rm and} \quad
K:=\left  \{\left (\begin{matrix} e^{iy}   &  0  \cr
                 0  & e^{-iy}    \cr   
\end{matrix} \right )  \ : \ y \in {\mathbb{R}}
\right \},$$
respectively.
Here we are interested in  $\,U$-equivariant
disc bundles over the complex affine 
quadric $\,Q^2 \cong U^{\mathbb{C}}/K^{\mathbb{C}}$. In the sequel  $\,K^{\mathbb{C}} \,$ is identified with $\,{\mathbb{C}}^*\,$
via the Lie group isomorphism given by
$$\, \left (\begin{matrix}  \zeta   &  0  \cr
                  0  & \zeta^{-1}    \cr   
\end{matrix} \right ) \to \zeta \,.$$
One checks that every holomorphic line bundle over
$\,Q^2\,$ is isomorphic to a homogeneous line bundle of the form (cf. Sect.\,2)
 $$\,L^m := U^{\mathbb{C}} \times _{\chi^m} {\mathbb{C}},$$
 where $\,m \in {\mathbb{Z}}\,$ and  the character $\,\chi^m:K^{\mathbb{C}} \to {\mathbb{C}}^* \,$ is defined by 
 $\, \chi^m(\zeta)=\zeta^m$.
Consider the symmetric decomposition $\,{\mathfrak{k}} \oplus {\mathfrak{p}}\,$ of the Lie algebra
$\,{\mathfrak{u}}\,$ of $\,U\,$ associated to the compact symmetric space $\,S^2 \cong U/K$,
and let $\,{\mathfrak{a}}\,$ be the  maximal abelian subalgebra  in $\,{\mathfrak{p}}\,$ generated by
a chosen element $\,H\,$ in  $\,{\mathfrak{p}}$.
By a  result of Mostow (\cite{Mos})
one has the decomposition $\,U^{\mathbb{C}}=U \exp (i{\mathfrak{a}}) K^{\mathbb{C}}$.
Then  a $\,U$-equivariant, punctured disc bundle in $\,L^m\,$
is uniquely defined by (cf. Sect. 3)
$$\Omega_{\rho}:=\{\,[g,z]\in L^m \ : \ |z||\zeta|^m\rho(h)<1  \,\}\,,$$
where $\,u \exp (i h H) \zeta^{-1}\,$ is a Mostow decomposition of 
$\,g\,$ and $\,\rho:{\mathbb{R}} \to {\mathbb{R}}^{>0}\,$ is an even, upper semicontinuous
function. Moreover one shows that $\,\Omega_\rho\,$ is Stein if and only if the function
$\,U^{\mathbb{C}} \to {\mathbb{R}}$, given by  
$\,g \to  |\zeta|^m \rho(h)$, 
is logarithmically plurisubharmonic (Prop.$\,$\ref{STEIN}).
 Warning: the function  $\, \ g \to \log |\zeta|\,$ is not plurisubharmonic. 

By acting fiberwise with a suitable element of $\,\exp (i{\mathfrak{k}})\,$ one can
 normalize $\,\Omega_\rho\,$ so that  $\,\rho(0)=1$.
Then, for all $\,m\in {\mathbb{Z}}\,$ one finds a maximal Stein, $\,U$-equivariant disc bundle 
$\,\Omega_{max}\,$ defined
 by $\,\rho_{max}(h):= (\cosh (2h))^{|m|/2}.$
It turns out that  the associated punctured disc bundle 
$\,\Omega_{max}^*$, which is obtained by removing the
zero section, is not Kobayashi 
hyperbolic. Indeed its universal covering admits  a proper $\,{\mathbb{C}}$-action. Moreover one shows  (Thm.$\,$\ref{MAIN})

\medskip
\noindent
{\it All other normalized, Stein, $\,U$-equivariant, punctured
disc bundles are
contained in $\,\Omega_{max}^*\,$ and are Kobayashi hyperbolic.}

\medskip
As an application we give a new proof of a known characterization
of the 3-dimensional, bounded symmetric domain of type IV
(Thm.$\,$\ref{APPLICATION}).

\bigskip
\noindent
{\bf Acknowledgement.}$\ $
I wish to thank Stefano Trapani for several helpful and pleasant
discussions.


\bigskip
\section{Line bundles over $\,Q^2$.}

\bigskip
Here  all holomorphic line bundles
over the affine complex quadric $\,Q^2\,$ are shown to 
be isomorphic to homogeneous line bundles of the form $\,U^{\mathbb{C}} \times_{\chi^m}{\mathbb{C}}$,
with  $\,\chi^m\,$ a character of $\,K^{\mathbb{C}}$.

Recall that the homogeneous bundle $\, U \times_K {\mathfrak{p}}:=
(U \times {\mathfrak{p}})/K$, where $\,K\,$ acts on $\,U \times {\mathfrak{p}}\,$ by 
$\, k \cdot (u,X):= (uk^{-1}, Ad_kX)$, 
can be identified with the tangent bundle of the compact symmetric space
$\,S^2 \cong U/K\,$
via the $\,U$-equivariant  diffeomorphism $\,U \times_K  {\mathfrak{p}} \to TS^2$,
defined by $\,[u,X] \to u_*(X)$. Here $\,{\mathfrak{p}}\,$ is identified with the tangent space of
$\,S^2\,$ at the base point via the differential of the canonical projection
$\,U \to S^2$. 

As a consequence of Mostow's decomposition
(\cite{Mos},$\,$Lemma$\,$4.1, cf.$\,$\cite{Las}, Thm. D 
and \cite{HeSc},$\,$Sect.$\,$9)
one also has a $\,U$-equivariant
identification of $\, U \times_K {\mathfrak{p}}
\to  U^{\mathbb{C}}/K^{\mathbb{C}} \cong Q^2\,$ given by $\,[u,X] \to u \exp (iX) K^{\mathbb{C}}$.
Hence one obtains an identification of
$\,U^{\mathbb{C}}/K^{\mathbb{C}}\,$ with the tangent bundle of $\,U/K$.

Realize the sphere $\,S^2 \cong U/K\,$ as the zero section of its tangent 
bundle via the immersion $\,\iota:U/K \to U^{\mathbb{C}}/K^{\mathbb{C}}\,$ defined
by $\,uK \to uK^{\mathbb{C}}$. Let 
$$\,B=\left \{ \left (\begin{matrix}  \zeta   &  0  \cr
                  \beta  & \zeta^{-1}      \cr    
\end{matrix} \right ) \ : \ \zeta \in {\mathbb{C}}^*, \ \beta \in {\mathbb{C}}\,\right \}\,$$
be the isotropy at $\,[0:1]\,$ with respect to the standard
linear $\,U^{\mathbb{C}}$-action on $\,{\mathbb{P}}^1$.
Consider the projection $\,\pi:U^{\mathbb{C}}/K^{\mathbb{C}} \to
U^{\mathbb{C}}/B\,$ given by $\,uK^{\mathbb{C}} \to uB$. One has the natural identifications
$\,U^{\mathbb{C}}/B \cong {\mathbb{P}}^1 \cong S^2\,$ and $\,\pi \circ \iota = Id_{S_2}$.
On the other hand the composition 
 $\,\iota \circ \pi\,$ is  the fiberwise projection onto the zero section
in the tangent bundle $\,U^{\mathbb{C}}/K^{\mathbb{C}}$, therefore it is homotopic
to  $\,Id_{U^{\mathbb{C}}/K^{\mathbb{C}}}$. It follows that $\,\iota\,$ is a homotopic equivalence
and consequently
$$\pi^*:H^2({\mathbb{P}}^1,{\mathbb{Z}}) \to H^2(Q^2,{\mathbb{Z}})$$

\smallskip
\noindent
is an isomorphism. Since $\,H^1({\mathbb{P}}^1,{\mathcal{O}}^*)=H^2({\mathbb{P}}^1,{\mathbb{Z}}) \,$ and
$\, H^1(Q^2,{\mathcal{O}}^*)=H^2(Q^2,{\mathbb{Z}})$, this gives an isomorphism
among the groups of holomorphic line bundles
$$\pi^*:Pic({\mathbb{P}}^1) \to Pic(Q^2)\,.$$
Now recall that
$$\,Pic({\mathbb{P}}^1)=\{\hat L^m:= U^{\mathbb{C}} \times_{\hat \chi^m} {\mathbb{C}} \ : \ m \in {\mathbb{Z}}\}\,,$$
where  $\,\hat \chi^m\,$ is the character of $\,B\,$ defined by 
$\, \left (\begin{matrix}  \zeta   &  0  \cr
                  \beta  & \zeta^{-1}    \cr    
\end{matrix} \right ) \to \zeta^m\,$ and $\,U^{\mathbb{C}} \times_{\hat \chi^m} {\mathbb{C}}\,$
is the quotient of $\,U^{\mathbb{C}} \times {\mathbb{C}}\,$ with respect to the $\,B$-action
defined by $\,b \cdot (g,z)= (gb^{-1}, \hat \chi^m(b)z)$. Indeed, since 
a homogeneous  bundle is uniquely defined by the isotropy representation on the 
fiber at the base point, one has 
$$\hat L^{m+n} =\hat L^{m} \otimes \hat L^{n}\,.$$
Then it is enough to note
that the generator $\,\hat L^{-1}\,$ of the group $\,\{\hat L^m \ : m \in {\mathbb{Z}}\}\,$
is biholomorphic to 
the tautological line bundle $\,T \subset {\mathbb{P}}^1 \times {\mathbb{C}}^2\,$
over $\,{\mathbb{P}}^1\,$ (cf.$\,$\cite{GrHa})
via the the map
$$\, [g,z] \to \left (g  \left [\begin{matrix}  0  \cr
                  1    \cr    
\end{matrix} \right ], zg \left (\begin{matrix}  0  \cr
                  1    \cr    
\end{matrix} \right ) \right ),$$
where the  action of $\,U^{\mathbb{C}}\,$ on $\,{\mathbb{C}}^2\,$ is
the standard linear one.

Finally consider the homogeneous bundles over
$\,Q^2\,$ of the form
$\,L^m := U^{\mathbb{C}} \times _{\chi^m} {\mathbb{C}}$, where  $\,\chi^m\,$
is the character on $\,K^{\mathbb{C}}\,$ defined by $\,\zeta \to \zeta^m$.
One has the canonical projection $\,U^{\mathbb{C}} \times_{\chi^m} {\mathbb{C}} \to U^{\mathbb{C}}/K^{\mathbb{C}}\,$ given
by $\, [g,z] \to gK^{\mathbb{C}}\,$ and 
 the bundle projection  $\,U^{\mathbb{C}} \times_{\chi^m} {\mathbb{C}} \to U^{\mathbb{C}} \times _{\hat \chi^m} {\mathbb{C}}
 \,$ defined by $\,[g,z] \to   [g,z]$.
 Moreover the  
diagram 
$$\begin{matrix} U^{\mathbb{C}} \times _{\chi^m} {\mathbb{C}}     &  \to   &  U^{\mathbb{C}} \times _{\hat \chi^m} {\mathbb{C}}  \cr
                                   &                 &           \cr    
            \downarrow  &                       & \downarrow \cr
                                   &                 &                \cr
           U^{\mathbb{C}}/K^{\mathbb{C}}  &     \buildover{\pi}  \to    &   U^{\mathbb{C}}/B  \,,\cr 
\end{matrix} $$
whose vertical maps are the canonical
 $\,U^{\mathbb{C}}$-equivariant projections,
is commutative. It follows that  $\,\pi^*(U^{\mathbb{C}} \times _{\hat \chi^m} {\mathbb{C}})
= U^{\mathbb{C}} \times _{\chi^m} {\mathbb{C}}\,$ which, by the above remarks
implies the following

\bigskip
\begin{prop}
\label{NORMAL}
Every holomorphic line bundle over the affine complex quadric 
$\,Q^2\,$ is isomorphic to a homogeneous line bundle 
$\,L^m := U^{\mathbb{C}} \times _{\chi^m} {\mathbb{C}}$, for some  $\,m \in {\mathbb{Z}}$.
\end{prop}


\bigskip
\section{Stein, $\,U$-equivariant disc bundles over $\,Q^2$}

\bigskip
Define a disc bundle in $\,L^m\,$ as a subdomain whose intersection with 
every fiber of the canonical projection onto $\,Q^2 \cong U^{\mathbb{C}}/K^{\mathbb{C}}\,$
consists of a disc of finite radius.
As a consequence of Mostow's decomposition, one has
$$U^{\mathbb{C}}= U \exp (i {\mathfrak{a}}) K^{\mathbb{C}}\,,$$
with $\,\mathfrak{a}\,$ a maximal abelian subalgebra of
$\,\mathfrak{p}\,$ (cf. Sect. 1).
Moreover every $\,U$-orbit in $\,U^{\mathbb{C}}/K^{\mathbb{C}}\,$ meets 
the ``slice" $\,\exp (i{\mathfrak{a}}) K^{\mathbb{C}}\,$ in an orbit of the Weyl group 
$\,W \cong {\mathbb{Z}}_2$. Here the non trivial element of the 
$\,W$-action is  given by reflection  in $\,{\mathfrak{a}}$. 

In particular $\,U  \backslash U^{\mathbb{C}}/K^{\mathbb{C}}\,$ is homeomorphic to 
$\,{\mathfrak{a}}/W\,$ and for every fixed $\,m \in {\mathbb{Z}}\,$ 
there is a one-to-one correspondence  among $\,U$-equivariant 
disc bundles in $\,L^m\,$ and even, upper semicontinuous,
positive functions on $\,{\mathfrak{a}}$. Namely, let $\,{\mathfrak{a}}\,$ be generated by 
$\,H:= \left (\begin{matrix}  0   &  -1  \cr
                  1 & \ 0  \cr    
\end{matrix} \right )$. Then an even, upper semicontinuous,
positive function $\,\rho: {\mathbb{R}} \to {\mathbb{R}}^{>0}\,$ defines a unique 
$\,U$-equivariant  disc bundle in $\,L^m\,$ by 
$$\,\Omega_{\rho}:= \{\, [g,z] \in U^{\mathbb{C}} \times_{\chi^m}
{\mathbb{C}} \ : \ |z| |\zeta |^m \rho(h) < 1 \,\}\,,$$
 where $\,u \exp (ihH) \zeta^{-1}\,$ is a Mostow decomposition of $\, g$.
 Let $\,U\times K\,$ act on $\,U^{\mathbb{C}}\,$ by $\,(u,k)\cdot g := ugk^{-1}$.
 
 It is easy to check that the $\,U \times K$-invariant function
 $\,U^{\mathbb{C}} \to {\mathbb{R}}^{> 0}, $ defined by $\,g \to |\zeta |^m \rho(h)$,
 does not depend on the chosen decomposition of $\,g\,$
 and consequently  $\,\Omega_{\rho}\,$ is well defined. 
Also note that  such a function defines a $\,U$-invariant
hermitian norm on $\,L^m$.

\bigskip
\begin{prop}

\noindent
{\rm (i)}  The $\,U$-equivariant  disc bundle $\,\Omega_{\rho}\,$ is Stein if and
only if the $\,U\times K$-invariant function $\, |\zeta |^m \rho(h)\,$
defined on $\,U^{\mathbb{C}}\,$ is  logarithmically plurisubharmonic. 

\noindent
{\rm (ii)} If $\,\Omega_{\rho}\,$ is Stein then $\, \rho\,$ is logarithmically
convex. In particular $\,\rho\,$
is continuous and realizes a minimum at zero.
\label{STEIN}
\end{prop}

\smallskip
\begin{proof} (i)
Let $\,\Pi:U^{\mathbb{C}} \times {\mathbb{C}}\ \to U^{\mathbb{C}} \times_{\chi^m}{\mathbb{C}}\,$ be the natural projection and 
$\,O_{\rho}:= \Pi^{-1}(\Omega_{\rho})$.
Then $\,O_{\rho}\,$ is a principal $\,{\mathbb{C}}^*$-bundle over $\,\Omega_{\rho}\,$
and, by a classical result of Serre (cf.$\,$\cite{MaMo}, Thm.$\,$4 and 6),
if  $\,\Omega_{\rho}\,$ is Stein so is $\,O_{\rho}$.
On the other hand  $\,\Omega_{\rho}\,$ is the 
quotient of $\,O_{\rho}\,$ with respect to the twisted $\,K^{\mathbb{C}}$-action.
Thus if  $\,O_{\rho}\,$ is Stein,  so is $\,\Omega_{\rho}\,$
by Theorem~5 in \cite{MaMo}.

Finally note that the generalized Reinhardt domain
$\,O_{\rho}= \{\,(g,z) \in U^{\mathbb{C}} \times {\mathbb{C}} \ : \ 
|z| < |\zeta|^{-m} \rho(h)^{-1} \}\,$ is Stein if and only if the 
function $\,U^{\mathbb{C}} \to {\mathbb{R}}$, given by $\,g \to
-\log (|\zeta|^{-m}\rho(h)^{-1})$, is plurisubharmonic (cf.~\cite{Vla},$\,$Sect.$\,$19.4).

\medskip 
\noindent
(ii) Let $\,f:{\mathbb{C}} \to U^{\mathbb{C}} \,$ be the holomorphic map
defined by $\,x +iy \to \exp (x+iy)H$.
By composing with the plurisubharmonic function 
$\,\log ( |\zeta|^{m}\rho(h))\,$ one obtains the  $\,{\mathbb{R}}$-invariant
function $\,{\mathbb{C}} \to {\mathbb{R}}$, given by $\,x+iy \to \log   \rho(y)$, whose 
subharmonicity is equivalent to convexity of   $\,\log \rho$.
The last part of the statement
follows from  elementary properties of convex, even functions on
$\,{\mathbb{R}}$. \qed
 \end{proof}
 
 \bigskip
 \begin{remark}
By \cite{AzLo},~Thm.$\,$1,~p.$\,$367,  the function  $\,\rho:{\mathbb{R}} \to {\mathbb{R}}^{>0}\,$
is logarithmically convex if and only if 
the $\,U\times K^{\mathbb{C}}$-invariant function on $\,U^{\mathbb{C}}$, defined  by 
$\, g \to \rho(h)$, is logarithmically plurisubharmonic. 
 \end{remark}

 \bigskip
 \begin{remark}
In the definition of a disc bundle one could allow the fibers to have infinite radius, i.e. 
the function $\,\rho\,$ to take values in $\,{\mathbb{R}}^{>0} \cup \{\infty\}$. Then,
for a Stein, $\,U$-equivariant disc bundle $\,\Omega_\rho$, 
the convexity of  $\,\rho\,$ would imply that either $\,\Omega_\rho= L^m\,$
or $\,\rho\,$ is real valued as in the above setting. That is,
no matter which definition one chooses, the above proposition describes
all proper, Stein, $\,U$-equivariant disc bundles over $\,Q^2$.
 \end{remark}


\bigskip
\section{Some coordinates}
\bigskip

For later use we introduce   some coordinates on the
double quotient $\,U \backslash U^{\mathbb{C}}/K$.
First consider the map
$$\,\Pi_1:U^{\mathbb{C}} \to U^{\mathbb{C}}\, , \quad \,g \to \sigma_U(g)^{-1}g\,,$$
\noindent
 where $\,\sigma_U:U^{\mathbb{C}} \to U^{\mathbb{C}}$, given by $\, g \to ^t\overline g^{-1}$,  is the 
 antiholomorphic involutive automorphism of $\,U^{\mathbb{C}}\,$
whose fixed point set is $\,U$.
Let $\,U\,$ act on $\,U^{\mathbb{C}}\,$ by left multiplication and
note that every fiber of $\,\Pi_1\,$ consists of a single $\,U$-orbit.
 Thus $\,\Pi_1(U^{\mathbb{C}})\,$ is set theoretically 
equivalent to $\, U\backslash U^{\mathbb{C}}\,$ and
$$\Pi_1:U^{\mathbb{C}} \to \Pi_1(U^{\mathbb{C}})$$

\smallskip
\noindent
is a realization of the quotient map.
Moreover, one checks  that 
$\,\Pi_1(U^{\mathbb{C}})\,$ consist of the connected component 
of  $\,\{\,g \in U^{\mathbb{C}} \ : \ \sigma_U(g)=g^{-1} \,\}$, explicitly given by
$${\mathcal Q} :=\left \{  \left ( \begin{matrix}  \ s  &  b   \cr
                   \overline b    &  t	  \cr 
                   \end{matrix} \right ) \, :
				   \,  s,t \in {\mathbb{R}}^{>0}, \ b \in {\mathbb{C}} \, \ {\rm and}\, \
				   st - |b|^2=1 \right \} .$$

\medskip
\noindent
Let us describe how the right $\,K$-action on $\,U^{\mathbb{C}}\,$
is transformed after applying  $\, \Pi_1$. 
An element of $\,K\,$ is given by $\,k=\exp (yC)\,$
for some real $\,y\,$ and $\,C:= \left ( \begin{matrix}  i  &  \ 0   \cr
                    0   &  -i	  \cr 
                    \end{matrix} \right )\,.$  Then one has 
$$\Pi_1(gk^{-1})
= \sigma_U(g \exp(-yC))^{-1}g\exp (-yC)=$$
$$\sigma_U(\exp(-yC))^{-1} \sigma_U(g)^{-1}g\exp (-yC)=k\Pi_1(g) k^{-1}$$
\noindent
Therefore $\, \Pi_1:U^{\mathbb{C}} \to {\mathcal Q}\,$ is $\,K$-equivariant, if one lets
$\,K\,$ act on $\,U^{\mathbb{C}}\,$ by right multiplication and on
$\,{\mathcal Q}\,$ by conjugation, i.e.

$$\exp (yC) \cdot \left ( \begin{matrix}  \ s  &  b   \cr
  \overline b    &  t	  \cr 
  \end{matrix} \right ) :=
\left ( \begin{matrix}  s  &  &  e^{2iy}b   \cr
\cr	
  e^{-2iy}\overline b    &  & t	  \cr
  \end{matrix} \right ),$$

\medskip
\noindent 
for every $\,y \in{\mathbb{R}}$.
In particular, after applying $\,\Pi_1$,
the $\,K$-action reads as
rotations on $\,b$. Let 
\smallskip
$$\,{\mathcal P}:=\{\, (s,t) \in {\mathbb{R}}^2 \ : \ st \ge 1\}$$ 

\smallskip
\noindent
and define 
$\,\Pi_2:{\mathcal Q} \to {\mathcal P}\,$ by
\smallskip
$$\left ( \begin{matrix}  \ s  &  b   \cr   
  \overline b    &  t	  \cr 
  \end{matrix} \right ) \to (s,t) \,.$$

\smallskip
\noindent
For every $\,(s,t) \in {\mathcal P}\,$ the inverse image 
$\,\Pi_2^{-1}(s,\, t)\,$ consists of 
a single $\,K$-orbit given by 
$$\left \{ \left ( \begin{matrix}  \ s  &  b   \cr   
  \overline b    &  t	  \cr 
  \end{matrix} \right ) \in {\mathcal Q} \ : \
  |b|^2=st-1 \right \}.$$
  
\nsmallskip
Hence $\,{\mathcal P}\,$
is a realization of 
the quotient $\,{\mathcal Q}/K \cong U \backslash U^{\mathbb{C}}/K\,$
and $\,(s,t)\,$ can be regarded as coordinates for  $\,U \backslash U^{\mathbb{C}}/K$.
Moreover the  composition map $\,\Pi_2 \circ \Pi_1\,$ is a realization of the
quotient map.

Now let $\,u \exp (i h H) \zeta^{-1}\,$ be a Mostow decomposition of an element
$\,g\,$ of $\,U^{\mathbb{C}}$, with $\,\zeta= e^{x+iy}$. One has
$$\Pi_2 \circ \Pi_1(g)= \Pi_1 \circ \Pi_2 \left (\exp (i h H)
\left ( \begin{matrix}  e^{-x}  &  0   \cr   
  0    &  e^{x}	  \cr 
  \end{matrix} \right ) \right)
  =$$
 $$ \Pi_2 \left ( \left ( \begin{matrix}  e^{-x}  &  0   \cr   
  0    &  e^{x}	  \cr
   \end{matrix} \right )
  \left ( \begin{matrix}  0  &  -2ih   \cr   
  2ih    &  0	  \cr 
  \end{matrix} \right )
  \left ( \begin{matrix} e^{-x}  &  0   \cr   
  0    &  e^{x}	  \cr 
  \end{matrix} \right ) \right )=
  (e^{-2x} \cosh 2h, e^{2x} \cosh 2h)\,.$$
 
 \nmedskip
  Then one can define the $\,U \times K$-invariant functions
  $\,|\zeta|= e^x\,$ and $\,h\,$ in terms of the coordinates $\,(s,t)\,$ on the 
  quotient $\,{\mathcal P} \cong U \backslash U^{\mathbb{C}}/K$. For this it is
  convenient to choose $\,h\,$ to be positive.
    
\bigskip
\begin{lemma}
\label{COORDIN} Let $\,u \exp (ihH) \zeta^{-1}\,$ be a Mostow
decomposition of an element $\,g\,$ in $\,U^{\mathbb{C}}$, with $\,h \ge 0$. 
\medskip
\item{\rm (i)} \ The $\,U \times K$-invariant function $\,g \to |\zeta|\,$ on $\,U^{\mathbb{C}}\,$ 
pushes down on $\,{\mathcal P}\,$ to 
$$|\zeta|= \root 4 \of {\frac{t}{s}}\,.$$

\medskip
\item{\rm (i)} \ The $\,U \times K$-invariant function $\,g \to h\,$ on $\,U^{\mathbb{C}}\,$ 
pushes down on $\,{\mathcal P}\,$ to 
$$h= \frac{1}{2} \,{\rm arccosh} \sqrt {st}\,.$$
\end{lemma}

\bigskip
\begin{remark}
\label{PLURI}
Note that if  $\,g= \left ( \begin{matrix}  z_1  &  z_3   \cr   
  z_2    &  z_4	  \cr 
  \end{matrix} \right )$, then
  $\,(s,t)= (|z_1|^2+|z_2|^2,|z_3|^2+|z_4|^2 )$.
  It follows that $\,\log t\,$ and $\,\log s\,$ are plurisubharmonic 
  functions on $\,U^{\mathbb{C}}$. 
  \end{remark}


\bigskip
\section{Hyperbolicity}

\bigskip

Given a $\,U$-equivariant disc bundle $\,\Omega_{\rho}\,$ as in
section 4, the associated punctured disc bundle 
$\,\Omega_{\rho}^*:= \{\,[g,z] \in U^{\mathbb{C}} \times_{\chi^m}
{\mathbb{C}}^* \ : \ |z| |\zeta |^m \rho(h) < 1 \,\}\,$ is obtained by removing the zero section and can be regarded as a particular  annular bundle
(cf.$\,$\cite{Aba}).
Here we first show that, up to $\,U$-equivariant biholomorphism, every 
Stein, $\,U$-equivariant disc bundle $\,\Omega_{\rho}\,$ over $\,Q^2\,$ is contained in
a maximal one, say $\,\Omega_{max}$.  Then we note that
the  universal covering of the associated punctured disc bundle $\,\Omega_{max}^*\,$
admits a 
proper $\,{\mathbb{C}}$-action. In fact  $\,\Omega_{max}^*\,$ turns out to be the unique 
 Stein, $\,U$-equivariant punctured 
disc bundle over $\,Q^2\,$ which is not Kobayashi hyperbolic.
We need the following lemma. Let $\,{\mathbb{C}}^*\,$ act on $\,L^m\,$ by
fiberwise multiplication.

\bigskip

\begin{lemma}
\label{DUALITY} There exists a  $\,{\mathbb{C}}^*$-equivariant 
biholomorphism  $\,\varphi:L^m \to L^{-m}\,$ which  maps  
$\,U$-equivariant, punctured disc bundles in $\,L^m\,$ 
onto $\,U$-equivariant, punctured disc bundles in  $\,L^{-m}$.
\end{lemma}

\smallskip
\begin{proof} Consider the basis of $\,{\mathfrak{u}}\,$ given  by 
$$\,C:= \left (\begin{matrix}  i   &  \ 0  \cr
                  0 & -i  \cr    
\end{matrix} \right )\,, \quad 
H:= \left (\begin{matrix}  0   &  -1  \cr
                  1 & \ 0  \cr    
\end{matrix} \right )\,, \quad W:= \left (\begin{matrix} 
 0   &  i  \cr
                  i & \ 0  \cr   
\end{matrix} \right )$$
and let $\,\hat \varphi:U^{\mathbb{C}} \to U^{\mathbb{C}}\,$ be the Lie group isomorphism associated
to the Lie algebra isomorphism mapping  $\,\{C,H,W\}\,$ into 
$\,\{-C,-H,W\}$. Extend $\,\hat \varphi\,$ to the isomorphism of 
$\,U^{\mathbb{C}} \times {\mathbb{C}} \,$ defined
by $\,(g,z) \to ( \hat \varphi(g),z)$.
Since $\,\hat \varphi(k) = k^{-1}\,$ for all $\,k \in K^{\mathbb{C}}$, one has
$$\,\hat \varphi(gk^{-1}, \chi^m(k)z)=
(\hat \varphi(g)k, \chi^m(k)z)=(\hat \varphi(g)(k^{-1})^{-1}, \chi^{-m}(k^{-1})z)\,.$$
This implies that $\,\hat \varphi\,$ pushes down to a biholomorphism 
$\,\varphi:L^m \to L^{-m}$. Moreover by 
construction $\,\hat \varphi(U)=U$, therefore every
$\,U$-invariant domain of $\,L^m\,$ is mapped onto a 
$\,U$-invariant domain of $\,L^{-m}$.

In order to avoid ambiguity, here we let $\,\Omega_{m,\rho}\,$ denote the $\,U$-equivariant disc bundle $\,\Omega_{\rho}\,$ contained in $\,L^m$. 
If $\,[g,z] \in \Omega_{m,\rho} $, with $\,g=u \exp (ihH) \zeta^{-1}$,
one has
$$\,  \varphi([g,z])= [\varphi(u) \exp(-ihH)(\zeta^{-1})^{-1} ,z]\,,$$
with
$\, |z||\zeta^{-1}|^{-m} \rho(-h) =|z||\zeta|^m \rho(h) <1$.
Thus $\,\varphi(\Omega_{m,\rho}) = \Omega_{-m,\rho}$, implying the
statement. \qed \end{proof}

 
 \bigskip
\begin{remark}
\label{BIO}
Since $\,\hat \varphi(K^{\mathbb{C}})=K^{\mathbb{C}}$, one can consider the  induced biholomorphism
$\,\hat \varphi:Q^2 \to Q^2\,$ and it is easy to check that 
 $\,L^m= \hat \varphi^*(L^{-m})\,$ for all $\,m>0$.
However recall that $\,L^m\,$ and $\,L^{-m}\,$ are not isomorphic as line bundles
over $\,Q^2$.
\end{remark}

\bigskip

If $\,m\not=0$, then  $\,U^{\mathbb{C}}\,$ acts transitively on $\,U^{\mathbb{C}} \times_{\chi^m} {\mathbb{C}}^*$
by $\,g \cdot [g',z] := [gg',z]\,$ and the isotropy at $\,[e,1]\,$ is
the cyclic group $\,\Gamma_m = \{\, \zeta \in K^{\mathbb{C}} \ : \ \zeta^m =1\}$.
Therefore one has a commutative diagram 
$$\begin{matrix}   U^{\mathbb{C}}      &    &    \cr
                            &                                   &               \cr 
              \downarrow &           \searrow \pi            &   \cr
                            &                                   &               \cr 
              U^{\mathbb{C}}/\Gamma_m    &     \cong  & U^{\mathbb{C}} \times_{\chi^m} {\mathbb{C}}^* \,, \cr
\end{matrix} $$
where $\,\pi\,$ is the orbit map  given by $\,\pi(g) =[g,1]$.
It follows that  $\,\widetilde \Omega_{\rho}^*:= \pi^{-1}(\Omega^*_{\rho}) =
 \{ g \in U^{\mathbb{C}} \ : \ |\zeta |^m \rho(h) < 1 \,\}\,$
is a covering of $\,\Omega^*_{\rho}\,$ with $\,m$-sheets.
In fact  it is the universal covering
of $\, \Omega_{\rho}^*$, since  it 
is homeomorphic to $\,U^{\mathbb{C}}$, which is simply connected.
Indeed $\,\widetilde \Omega_{\rho}^*\,$ 
itself can be regarded 
 as a disc bundle over $\,Q^2\,$ and one can apply a
 suitable fiberwise radial dilatation deforming $\,\widetilde \Omega_{\rho}^*\,$
 onto  $\,U^{\mathbb{C}}$.  

For every $\,m \in {\mathbb{Z}}\,$ let $\,\Omega_{max}^*\,$ be the $\,U$-equivariant, punctured disc bundle in $\,L^m\,$ associated 
to $\,\rho_{max}:{\mathbb{R}} \to {\mathbb{R}}^{>0}\,$ defined by $\,\rho_{max}(h):= (\cosh 2h)^{|m|/2}$.

\bigskip

\begin{prop}
\label{MAX}
The  $\,U$-equivariant, punctured disc bundle $\,\Omega_{max}^*\,$
is Stein and  its
universal covering  admits
a proper $\,{\mathbb{C}}$-action. In particular $\,\Omega_{max}^*\,$ is not
Kobayashi hyperbolic.
\end{prop}

\smallskip
\begin{proof} For $\,m=0\,$ one has $\,\Omega^*_{max}= Q^2 \times
\Delta^*\,$ and the statement follows by considering
the action on $\,Q^2\,$ of any one parameter
subgroup in $\,U^{\mathbb{C}}$. Next, by Lemma \ref{COORDIN} one has
$\,|\zeta|= \root 4 \of {\frac{t}{s}}\,$ and 
$$\,h= \frac{1}{2}\, {\rm arccosh} \, \sqrt{st} = \frac{1}{2} {\rm arccosh}
\,  e^{\frac{1}{2}(\log s + \log t)}\,.$$ 
Define
$\,\theta:{\mathbb{R}}^{\ge 0} \to {\mathbb{R}}^{\ge 0}\,$ by $\,\theta(\tau):= \log  \rho
(\frac{1}{2} {\rm arccosh} \,  e^{\frac{\tau}{2}})$. Then
$$\log (|\zeta|^m \rho(h)) =m \log |\zeta| + \log  \rho(h) = 
\frac{m}{4} (\log t - \log s) + \theta (\log t + \log s) \,=$$ 
$$ \theta (\log t + \log s) - \frac{m}{4} (\log t + \log s) +\frac{m}{2} \log t.  $$
Assume that $\,m>0\,$ and fix $\,\rho_{max}(h)= (\cosh 2h)^{m/2}$,
which corresponds to  $\,\theta_{max} (\tau)=\frac{m}{4}\tau$.
Then the above equation implies that 
$\,\log |\zeta|^m \rho_{max}(h)=\frac{m}{2} \log t$, which is 
plurisubharmonic by  Remark \ref{PLURI}.
Therefore $\,\Omega_{max}\,$
 is Stein by Prop. \ref{STEIN} and so is the associated
 punctured disc bundle $\,\Omega^*_{max}$. 
 
 Finally note that the function $\,U^{\mathbb{C}} \to {\mathbb{R}}^{>0}\,$ given by 
 $\,g \to t\,$ is invariant
 with respect to the proper $\,{\mathbb{C}}$-action on $\,U^{\mathbb{C}}\,$ defined by
 (cf. Rem. \ref{PLURI})
 $$w\cdot g:= g\left (\begin{matrix}  1 &  0  \cr
                  w  & 1    \cr    
\end{matrix} \right ).$$
 Thus   the universal covering 
 $\,\widetilde \Omega^*_{max}=\{\,g \in  U^{\mathbb{C}} 
 \ : \   \frac{m}{2} \log t < 0 \, \}\,$ is a $\,{\mathbb{C}}$-invariant subdomain 
 of $\,U^{\mathbb{C}}$, proving the statement for $\,m>0$. A similar argument
 (or use Lemma \ref{DUALITY}) applies to the case when $\,m<0$. \qed
\end{proof}

\bigskip
Note that the fiberwise multiplication by $\,\rho(0)\,$ on
$\,U^{\mathbb{C}} \times_{\chi^m} {\mathbb{C}}$, given by $\,[g,z] \to
[g, \rho(0)z]$, 
maps $\,\Omega_\rho\,$ biholomorphically and $\,U$-equivariantly
onto $\,\Omega_{\rho/\rho(0)}$. It follows that one can always normalize 
$\,\Omega_\rho\,$ so that $\,\rho(0)=1$.

\bigskip

\begin{theorem}
\label{MAIN}

\nmedskip
\item{\rm (i)} Every Stein, normalized,
$\,U$-equivariant, disc bundle $\,\Omega_\rho\,$ over $\,Q^2\,$
 is contained in $\,\Omega_{max}$, and
\item{\rm (ii)} if $\,\Omega_\rho\,$ does not coincide with
$\,\Omega_{max}$,  then the associated 
Stein, punctured disc bundle $\,\Omega^*_\rho\,$
is Kobayashi Hyperbolic.
\end{theorem}

\smallskip
\begin{proof} 
First assume $\,m=0\,$ and let  
$\,\Omega_{\rho}=\{\, (gK^{\mathbb{C}},z) \in Q^2 \times {\mathbb{C}}  \ : \ |z|\rho(h)<1 \,\}\,$ be Stein.
Since $\,\rho(0)=1\,$ and $\,\rho\,$ is logarithmically convex by
Prop. \ref{STEIN}, it follows that 
$\,\Omega_{\rho} \subset Q^2 \times \Delta = \Omega_{max}$, proving (i). 
For (ii) consider the associated (Stein) punctured disc bundle $\,\Omega_{\rho}^*$  
and the projection $\,\pi:\Omega^*_\rho
\to \Delta^*\,$ onto the second factor. Note that 
if $\,\rho\,$ is non constant then $\,\rho(h)^{-1} \to 0\,$ as
$\,h \to \infty$. As a consequence for every relatively compact 
domain $\,A\,$ of $\,\Delta^*\,$ the preimage $\,\pi^{-1}(A)\,$ is 
contained in the product  $\,U \exp(IiH)K^{\mathbb{C}} \times A$, for some relatively
compact interval $\,I\,$ in $\,{\mathbb{R}}\,$. In particular $\,\pi^{-1}(A)$, being relatively compact in a Stein manifold,  is Kobayashi hyperbolic and
so is  $\,\Omega^*_\rho\,$ by Thm.~3.2.15 in \cite{Kob}.

If $\,m\not=0\,$ we  prove the inclusion in (i) for the universal coverings 
$\,\widetilde \Omega_\rho\,$ and $\,\widetilde \Omega_{max}$. 
As a consequence of Lemma \ref{DUALITY}
it is enough to consider the case  $\,m > 0$. Recall that 
$$\,\widetilde \Omega_{max}=\{\, g \in U^{\mathbb{C}} \ : \ \frac{m}{2} \log t < 0 \,\}\,.$$
 Note that  
 $$\,\widetilde \Omega_\rho =\{\, g \in U^{\mathbb{C}} \ : \  \delta 
(\log s + \log t) +\frac{m}{2} \log t  < 0 \,\}\,,$$
 where $\,\delta:[0, \infty) \to {\mathbb{R}}\,$ is defined by 
 $$\delta(\tau):= \theta (\tau)-\frac{m}{4}\tau = \log 
 \rho (\frac{1}{2} {\rm arccosh}\, (e^{\tau/2})) - \frac{m}{4}\tau \,.$$
Since $\,\delta(0)=0$, in order to prove
(i)  it is enough to show that $\,\delta\,$ is increasing.
Indeed one has

\nmedskip
{\it Claim.} $\ $The function $\,\delta\,$
is increasing. Moreover, if   $\,\delta \not \equiv 0$
then $\,\delta(\tau) \to \infty\,$ as $\,\tau \to \infty$.

\nmedskip
 {\it Proof.} $\ $ Since $\,\Omega_\rho\,$ is Stein, the $\,U \times K$-invariant 
 function  $\,U^{\mathbb{C}} \to {\mathbb{R}}^{>0}$, given by 
  $\,|\zeta|^m\rho(h)= \delta (\log s + \log t) +\frac{m}{2} \log t$, 
 is plurisubharmonic (cf. Proposition \ref{STEIN}).
 Then, by composing with the 
holomorphic map $\,{\mathbb{C}} \to U^{\mathbb{C}}$, defined by 
$$\,x+iy \to 
\left (\begin{matrix}  1   &  0  \cr
                  e^{x+iy}  & 1    \cr    
\end{matrix} \right )\,,$$
one obtains an
subharmonic, $\,i{\mathbb{R}}$-invariant function, namely
$\,x+iy \to \delta(\log (1+e^{2x}))$.
It follows that  the function $\,x \to \delta(\log (1+e^{2x}))\,$ is convex. 
Then it is necessarily 
increasing, since it converges to $\,0\,$  as $\,x \to -\infty$.
Furthermore  $\,x \to  \log (1+e^{2x})\,$
is strictly increasing, therefore $\,\delta\,$ is also increasing, as claimed.
Finally note that if $\,\delta \not \equiv 0$, then $\,x \to \delta(\log (1+e^{2x}))\,$
is non constant, convex and increasing. Then 
necessarily $\,\delta(\tau) \to \infty\,$ as $\,\tau \to \infty$, concluding the
proof of the claim.

\nmedskip

For (ii) note that by Theorem~3.2.8 in \cite{Kob}  the Stein, punctured
disc bundle $\,\Omega^*_\rho\,$ is Kobayashi hyperbolic if and only if
its covering $\,\widetilde \Omega^*_\rho \subset U^{\mathbb{C}}\,$ is hyperbolic.
Assume as above that $\,m > 0 \,$ and consider the projection
$$P: \widetilde \Omega^*_\rho \to {\mathbb{C}}^2 \setminus \{(0,0)\}\,,
\quad \quad  \left (\begin{matrix}  z_1 &  z_3  \cr
                  z_2  & z_4    \cr    
\end{matrix} \right )  \to (z_3,z_4) .$$
Since $\,\delta \ge0\,$ and $\, \delta 
(\log s + \log t) +\frac{m}{2} \log t <0\, $ on $\,  \widetilde \Omega^*_\rho \,$ it follows that $\,t=|z_3|^2 +|z_4|^2<1\,$
and consequently $\,P( \widetilde \Omega^*_\rho)$, being contained in the punctured unit ball $\,\mathbb{B}^*_1(0,0)\,$ of $\,{\mathbb{C}}^2$,
 is Kobayashi hyperbolic.
Then, by Thm.~3.2.15 in \cite{Kob}, in order to show that  $\,  \widetilde
\Omega^*_\rho \,$
is Kobayashi hyperbolic it is sufficient to show that for every fixed $\,(z_3,z_4)\,$
in  $\,P( \widetilde \Omega^*_\rho)\,$ there exists $\,\varepsilon\,$ small enough  such 
that $\,P^{-1}(\mathbb{B}_\varepsilon (z_3,z_4))\,$ is Kobayashi hyperbolic.
Here $\,\mathbb{B}_\varepsilon (z_3,z_4)\,$ denotes the ball centered in $\, (z_3,z_4)\,$
of radius $\, \varepsilon$.
Choose $\,\varepsilon\,$ such that $\,\mathbb{B}_\varepsilon (z_3,z_4)\,$ is relatively compact
in  $\,\mathbb{B}^*_1 (0,0)$. 
Then there exists a real, positive constant $\,C\,$ such that
$\, -C < \log t\,$ and consequently $\, \delta( \log s + \log t)  < \frac{m}{2}C\,$
on $\,P^{-1}(\mathbb{B}_\varepsilon (z_3,z_4))$.
Since by assumption
$\,\rho \not \equiv \rho_{max}$, i.e. $\,\delta \not \equiv 0$, the above
claim implies that $\,\delta(\tau) \to \infty\,$ as $\,\tau \to \infty$.
It follows that $\,\log s + \log t < D\,$ for some real constant $\,D$. 
Hence $\,\log s < D+C \,$ and consequently $\,s= |z_1|^2 +|z_2|^2 \,$
is bounded.
This implies that  $\,P^{-1}(\mathbb{B}_\varepsilon (z_3,z_4))\,$ is contained in the product
of two balls in $\,{\mathbb{C}}^4$, therefore it is Kobayashi hyperbolic. 
\qed
 \end{proof}
 
 
\bigskip
\begin{remark}
\label{HYPERCONVEX}
Note that the Stein, $\,U$-equivariant, punctured  disc bundles 
$\,\Omega_\rho^*\,$ are not hyperconvex,
in the sense of \cite{Sth}. Assume by contradiction that
there exists a bounded plurisubharmonic exhaustion $\,\varphi\,$
defined on $\, \Omega_\rho^*$. 
Since every fiber $\,F\,$ is closed in $\,\Omega_\rho^*$,
the restriction $\,\varphi |_F\,$ of $\,\varphi\,$ to $\,F\,$
is a subharmonic exhaustion. In particular 
$\,\varphi |_F\,$ is not constant. However $\,F\,$ is biholomorphic
to a punctured disc and  $\,\varphi |_F\,$ is bounded,
therefore $\,\varphi |_F\,$ extends to a bounded, subharmonic
function on the whole disc with a maximum at the origin. 
Hence $\,\varphi |_F\,$ is constant, giving a contradiction.
\end{remark}

\bigskip
\noindent
For later use we note the following fact.

 
 \bigskip
\begin{lemma}
\label{SEZIONE} Let $\,\Omega_\rho \subset L^m\,$ be a Stein, $\,U$-equivariant 
disc bundle over $\,Q^2$. 
If $\,m\not =0\,$ then every automorphism of $\,\Omega_\rho\,$ leaves 
the zero section invariant.
\end{lemma}
 
\begin{proof} 
Note that if $\,p\,$ belongs to the zero section $\,Z \cong U^{\mathbb{C}}/K^{\mathbb{C}}$,
then for every $\,X\,$ in the 2-dimensional tangent space
$\,T_pZ \cong \mathfrak p^{\mathbb{C}}\,$ there exists an entire curve through $\,p\,$
and tangent to $\,X$. Namely, $\, \exp ({\mathbb{C}} X) \cdot p$.
Then it is enough to show that for $\,p \in \Omega^*_\rho\,$ the subspace
of the elements of $\,T_p\Omega_\rho\,$ with this property is lower dimensional.

For this consider the free action of the cyclic group 
 $\,\Gamma_m \subset K^{\mathbb{C}} \cong {\mathbb{C}}^*\,$ on the punctured unit ball $\,\mathbb{B}_1^*(0,0)\,$
 in $\,{\mathbb{C}}^2\,$  given by 
$\,\gamma \cdot (z,w):=(\gamma z, \gamma w)$. Let  
$\,P:\Omega_\rho^* \to \mathbb{B}_1^*(0,0)/\Gamma_m\,$ be the  projection defined by
(cf. the proof of Thm. \ref{MAIN})
$$\left [ \left (\begin{matrix}  z_1 &  z_3  \cr
                  z_2  & z_4    \cr    
\end{matrix} \right ),1 \right ]  \to [z_3,z_4] $$
and let  
$\,\iota: \mathbb{B}_1^*(0,0)/\Gamma_m \to \Delta^3\,$ be the 
injective holomorphic map
defined by $\,[z,w] \to (z^m, z^{m-1}w, w^m)$.

For an  entire curve $\,f:{\mathbb{C}} \to \Omega_\rho\,$ through $\,p \in \Omega^*_\rho\,$
the inverse image  $\,f^{-1}(Z)\,$ is a discrete set. Moreover
 the composition
$\,\iota \circ P \circ f|_{{\mathbb{C}} \setminus f^{-1}(Z)}:{\mathbb{C}} \setminus f^{-1}(Z)\to \Delta^3\,$ 
defines a bounded holomorphic map. Thus it extends to a bounded
holomorphic function on $\,{\mathbb{C}}\,$ which, by Liouville's theorem
is constant. It follows that $\,f({\mathbb{C}})\,$ is contained in the one
dimensional fiber $\,P^{-1}(P(p))\,$ of $\,P$, which proves the
statement. \qed
\end{proof}


\bigskip
\section{A characterization}

\bigskip
 
A recent classification of holomorphic actions of classical 
simple, real Lie groups
by Huckleberry and Isaev
applies to show that      
the bounded symmetric domain  $\,SO(3,2)/(SO(3) \times SO(2))\,$
is characterized among Stein manifolds by its complex dimension
and by its automorphism group (see$\,$Thm.$\,$8.1$\,$in$\,$\cite{HuIs}).
As an application of Theorem \ref{MAIN}  we present a
different proof of this fact. Here we follow
the strategy pointed out  in \cite{GIL}, where  higher
dimensional bounded symmetric domains of type IV were considered.
We need a preparatory lemma. 
For notations and definitions we refer to \cite{GIL}.

 
 \bigskip
\begin{lemma}
\label{OMOGENEO} {\rm (cf.~Prop.~4.7~in~\cite{GIL})}
Let $\,X\,$ be a 3-dimensional Stein manifold such that 
$\, Aut(X) \,$ is isomorphic to $\,SO(3,2)$.
Assume that $\,X\,$ contains a minimal $\,SO(3) \times SO(2)$-orbit 
of dimension 3 which is $\,SO(3)$-homogeneous. Then $\,X$ is biholomorphic to
a $\,U \times K$-invariant domain
in $\,U^{\mathbb{C}}/\Gamma_m$.
\end{lemma}
 
\begin{proof} 
Let $\,M =(SO(3,{\mathbb{R}}) \times SO(2,{\mathbb{R}}))/H\,$ be the minimal 3-dimensional 
orbit. Then the connected component $\,H^e\,$ of the
isotropy subgroup  $\,H\,$ at $\,e\,$ is 1-dimensional and there exists an isomorphism
$\,SO(2,{\mathbb{R}}) \to H^e$, say $\,t \to (\varphi (t),\psi (t))$.
By   (ii) of Lemma~4.6 in \cite{GIL}
the $\,SO(2,{\mathbb{R}})$-action on $\,M\,$ is  free, therefore the homomorphism
$\,\varphi\,$ is injective.
Up to Lie group isomorphism we may assume that 
$\,\varphi (SO(2,{\mathbb{R}}))\,$ is the one parameter subgroup of $\,SO(3,{\mathbb{R}})\,$ generated
by the element
$$\,\left (\begin{matrix}  0 &  \ 0 & \ \,0  \cr
                   0 & \  0 & -1  \cr
                    0 &\   1 & \ \, 0  \cr 
\end{matrix} \right )  \,$$
belonging to the Lie algebra of $\,SO(3,{\mathbb{R}})$.
By assumption $\,M=SO(3,{\mathbb{R}})/F$, where $\,F:=H \cap SO(3)\,$ is finite.
Since
$\,F\,$ is a subgroup of $\,H$, it
normalizes $\,H^e$. As a consequence $\,F\,$ is contained in the 
normalizer of $\,\varphi (SO(2,{\mathbb{R}}))$, which is given by  
$\,\varphi (SO(2,{\mathbb{R}}))\cup \gamma \varphi (SO(2,{\mathbb{R}})) \cong O(2,{\mathbb{R}})$,
where $\,\gamma:= diag(-1,-1,1)$. However, if $\,F\,$ contains an
element of the form $\,\gamma \varphi(t')$, then for all $\,t \in SO(2,{\mathbb{R}})\,$
one has $\,(\gamma \varphi(t')\varphi (t) \varphi(t')^{-1}
\gamma^{-1},\,\psi (t)) = (\varphi (t)^{-1},\psi (t)) \in H\,$ and consequently
$\,(\varphi (t)^{-1},\psi (t))(\varphi (t),\psi (t))=(e,\psi (t)^2) \in H$.
Since the $\,SO(2,{\mathbb{R}})$-action on $\,M\,$ is  free
((ii)$\,$of$\,$Lemma$\,$4.6$\,$in$\,$\cite{GIL}), this implies that the homomorphism $\,\psi\,$ is trivial
and $\,M=SO(3,{\mathbb{R}})/ O(2,{\mathbb{R}}) \times SO(2,{\mathbb{R}})$, contradicting the transitivity
of the $\,SO(3,{\mathbb{R}})$-action.
Hence $\,F\,$ is  a cyclic subgroup
of $\,\varphi (SO(2,{\mathbb{R}}))$. 
Consider the commutative diagram
$$\begin{matrix}   SO(3,{\mathbb{R}})      &    &    \cr
                            &                                   &               \cr 
              \downarrow &           \searrow \Psi         &   \cr
                            &                                   &               \cr 
             SO(3,{\mathbb{R}})/F   &     \cong  & (SO(3,{\mathbb{R}}) \times SO(2,{\mathbb{R}}))/H = M \,, \cr
\end{matrix} $$
where the surjective orbit map 
$\,\Psi\,$ is defined by $\,\Psi(g)=[g,e]$.
Let $\,SO(3,{\mathbb{R}}) \times SO(2,{\mathbb{R}})\,$ act on $\,SO(3,{\mathbb{R}})\,$
 by $\,(g',t)\cdot g:=
g'g \varphi(t)^{-1}\,$ and naturally on $\,M\,$ (i.e., by  left
$\,SO(3,{\mathbb{R}}) \times SO(2,{\mathbb{R}})$-action).
One has 
$$\,\Psi(g'g \varphi(t)^{-1})=[g'g \varphi(t)^{-1},\psi (t) \psi (t)^{-1}]=[g'g,\ \psi (t)]=(g',\psi (t))
\cdot \Psi(g)\,.$$
Now recall that $\,X\,$ is biholomorphic to an
$\,SO(3,{\mathbb{R}}) \times SO(2,{\mathbb{R}})$-invariant domain 
in the complexified orbit
$\,(SO(3,{\mathbb{C}}) \times SO(2,{\mathbb{C}}))/H^{\mathbb{C}}\,$
(cf. the beginning of Sect.~4.1 in \cite{GIL}) and extend the isomorphism in the above 
diagram to $\,SO(3,{\mathbb{C}})/F \to (SO(3,{\mathbb{C}}) \times SO(2,{\mathbb{C}}))/H^{\mathbb{C}}$.
Then,  the analytic continuation principle
and the above equivariance relation
 imply that the manifold
$\,X\,$ is biholomorphic to  an 
$\,SO(3,{\mathbb{R}}) \times SO(2,{\mathbb{R}})$-invariant domain
in $\,SO(3,{\mathbb{C}})/F$.

Finally let $\,\Pi:U^{\mathbb{C}} \to SO(3,{\mathbb{C}})\,$ be a universal covering of 
$\,SO(3,{\mathbb{C}})\,$ which maps  $\,U\,$  onto \,$ SO(3,{\mathbb{R}})\,$
and $\,K\,$ onto $\,H^e$. Then the finite subgroup
 $\,\Pi^{-1}(F)\,$ of $\,K\,$ is cyclic
  and $\, SO(3,{\mathbb{C}})/F\,$ is equivariantly biholomorphic to
  $\,U^{\mathbb{C}}/\Pi^{-1}(F)$, implying the statement. \qed
\end{proof}

 
\bigskip
\begin{theorem}
\label{APPLICATION}
Let $\,X\,$ be a 3-dimensional Stein manifold such that 
$\, Aut(X) \,$ is isomorphic to $\,SO(3,2)$.
Then $\,X\,$ is biholomorphic to
 the bounded symmetric domain $\,SO(3,2)/(SO(3) \times SO(2))$.
\end{theorem}
 
\begin{proof} 
If the maximal compact subgroup
$\,SO(3) \times SO(2)\,$ has a fixed point in $\,X$, then 
the statement follows from Prop.~3.1 in \cite{GIL}.
 
So let us  assume by contradiction that 
$\,SO(3) \times SO(2)\,$ has no fixed points in $\,X$. Then, as a consequence
of Lemma \ref{OMOGENEO} above,  Prop.~4.8~and~4.10 in \cite{GIL}, the manifold $\,X\,$ is biholomorphic to 
a $\,U$-invariant domain in a line bundle either over the complex affine
quadric $\,Q^2\,$ or over $\,Q^2/{\mathbb{Z}}_2$. Here we allow finite ineffectivity in
order to replace the action of $\,SO(3,{\mathbb{R}})\,$ with the action of 
 its universal covering $\,U=SU(2)$.
If the base is $\,Q^2$, the line bundle is given by $\,L^m:=U^{\mathbb{C}} \times_{\chi^m} {\mathbb{C}}^*$, for some $\,m \in {\mathbb{Z}}$, with projection $\,p:L^m \to Q^2\,$ given by $\,[g,z] \to gK^{\mathbb{C}}\,$
(cf. Sect.2). 
We distinguish several cases.

If $\,p(X)\,$ does not coincide with $\,Q^2$, then $\,p(X)\,$
is Kobayashi hyperbolic 
and so is $\,X\,$ by  Thm.~3.2.15 in \cite{Kob} 
(cf.~the~proof~of~Thm.~5.5~in~\cite{GIL}).
Then,
as a consequence of  Prop.~3.2 in \cite{GIL} the group 
$\,SO(3) \times SO(2)\,$ has a fixed point,
giving a contradiction.

If $\,p(X)=\,Q^2\,$ and $\,m=0$, analogous arguments as in the
proof of Thm.~5.5 in \cite{GIL} imply that either $\,X\,$
is Kobayashi hyperbolic or $\, Aut(X)\,$ 
is infinite dimensional, giving again a  contradiction.

If $\,p(X)=\,Q^2\,$ and $\,m\not =0$, then one checks that 
$\,X\,$ is biholomorphic
either to a disc bundle $\,\Omega_\rho\,$ or to a
punctured disc bundle $\,\Omega_\rho^*$. As a 
consequence of Lemma \ref{SEZIONE},
in both cases
$\,SO(3,2)\,$ acts on $\,\Omega_\rho^*$.
If  $\,\Omega_\rho^*  \not = \Omega_{max}^* $, then $\,\Omega^*_\rho\,$
is Kobayashi hyperbolic by Theorem \ref{MAIN} and one obtains a contradiction as 
 above.
 
 In the case when 
 $\,\Omega_\rho^*   = \Omega_{max}^* \,$
consider the projection
 $\,P:\Omega_{max}^* \to {\mathbb{B}}_1^*(0,0)/\Gamma_m\,$ 
  introduced
in the proof of  Lemma  \ref{SEZIONE} and note that
every fiber $\,F\,$ of $\,P\,$
 is biholomorphic to 
$\,{\mathbb{C}}$. Then hyperbolicity of $\,{\mathbb{B}}_1^*(0,0)/\Gamma_m\,$ implies
that for every $\,g \in SO(3,2)\,$ the composition  $\,P \circ g|_F\,$
is constant. That is, $\,g\,$ maps fibers to fibers and consequently
the $\,SO(3,2)$-action on   $\,\Omega_{max}^*\,$ pushes down to 
an action on $\,{\mathbb{B}}_1^*(0,0)/\Gamma_m$. By hyperbolicity of
$\,{\mathbb{B}}_1^*(0,0)/\Gamma_m\,$
such an action is necessarily proper and consequently every isotropy
subgroup is contained in a copy of the maximal one. It follows
that  the minimal real
dimension of every $\,SO(3,2)$-orbit in
$\,{\mathbb{B}}_1^*(0,0)/\Gamma_m\,$ is six. Since  $\,{\mathbb{B}}_1^*(0,0)/\Gamma_m\,$
is a complex 2-dimensional manifold, this gives a contradiction. 

Similar arguments apply to the case when
$\,X\,$ is biholomorphic to 
a $\,U$-invariant domain in a line bundle over  $\,Q^2/{\mathbb{Z}}_2\,$
 and we omit the details. \qed
\end{proof} 
 
 \nbigskip

\hfill
  

\end{document}